\documentclass{amsart}
\usepackage{amsfonts}
\usepackage{amsmath,amssymb}
\usepackage{amsthm}
\usepackage{amscd}

\usepackage{graphicx}
{
\theoremstyle{remark}
\newtheorem{Def}{{\rm Definition}}
\newtheorem{Ex}{{\rm Example}}
\newtheorem{Rem}{{\rm Remark}}

}
\newtheorem{Cor}{Corollary}
\newtheorem{Fact}{Fact}
\newtheorem{Prop}{Proposition}
\newtheorem{Thm}{Theorem}

\newtheorem*{Prob*}{Problem}
\begin{document}
\setcounter{figure}{0}
\title[An elementary study of homology groups of Reeb spaces of fold maps]{An elementary study on realizable changes of homology groups of Reeb spaces of fold maps by fundamental surgery operations}
\author{Naoki Kitazawa}
\address{Institute of Mathematics for Industry, Kyushu University, 744 Motooka, Nishi-ku Fukuoka 819-0395, Japan}
\email{n-kitazawa@imi.kyushu-u.ac.jp}
\maketitle
\begin{abstract}
In the singularity and differential topological theory of Morse functions and higher dimensional versions or {\it fold} maps and application to algebraic and differential topology of manifolds, constructing explicit fold maps and investigating their source manifolds is fundamental, important and difficult. The author has introduced surgery operations ({\it bubbling operations}) to fold maps, motivated by studies of Kobayashi, Saeki etc. since 1990 and has explicitly shown that homology groups of {\it Reeb spaces} of maps constructed by iterations of these operations are flexible in several cases. Such operations seem to be strong tools in construction of maps and precise studies of manifolds. 

More precisely, the author has also noticed that the resulting groups are represented as direct sums of the original homology groups and suitable finitely generated commutative groups. 

The {\it Reeb space} of a map is the space of all connected components of inverse images of the maps. Reeb spaces inherit fundamental invariants of the manifolds such as homology groups etc. much in simple cases as polyhedra whose dimensions are equal to those of the target spaces.

This paper is on a new explicit study of changes of homology groups of Reeb spaces of fold maps by the surgery operations. We present explicit changes obtained by an approach via elementary theory of sequences of numbers and fundamental continuous or differentiable functions.         

\end{abstract}

\section{Introduction}
\label{sec:1}
Throughout this paper, manifolds, maps between them, bundles whose fibers are manifolds etc. are fundamental geometric objects. They are smooth and of the class $C^{\infty}$ except cases where several differentiable functions on intervals are considered, where extra special explanations are done etc.. We give additional explanations on geometric objects and notions.

We call a bundle whose fiber is a topological space $X$ an {\it $X$-bundle}. A {\it linear} bundle whose fiber is a standard closed disc (an unit disc) or a standard sphere (an unit sphere) means a bundle whose
 structure group acts on the fiber linearly.

We also explain about homotopy spheres. A homotopy sphere is called an {\it exotic} sphere if it is not diffeomorphic to a standard sphere and an {\it almost-sphere} if it is obtained by gluing two copies of a standard closed sphere on the boundaries. Every homotopy sphere except $4$-dimensional exotic spheres, being undiscovered, is an almost-sphere.

Last, for a smooth map, we call the set of all {\it singular} points, defined as points such that at the points the ranks of the differentials drop, the {\it singular set} and the image of the singular set the {\it singular value set}. For the map, we call the value of a singular point the {\it singular value}, a point not in the singular value set a {\it regular value} and the complement of the singular value set the {\it regular value set}. 

Commutative groups and more generally, modules over principal ideal domains are fundamental and important algebraic objects in the present paper. For a finitely generated module $G$ over a principal ideal domain $R$, we denote the rank by ${\rm rank}_R G$.

\subsection{Historical backgrounds on studies of the present paper and fundamental explanations on fundamental tools}
\subsubsection{Fold maps}
{\it Fold} maps are higher dimensional versions of Morse functions and fundamental and important tools in the theory of Morse functions and higher dimensional versions and application to algebraic and differential topological
 theory of manifolds: in other words the global singularity theory.
\begin{Def}
\label{def:1}
Let $m \geq n$ be positive integers. A smooth map between an $m$-dimensional smooth manifold without boundary and an $n$-dimensional smooth manifold without boundary is said to be a {\it fold map} if at each singular point $p$, the form is
$$(x_1, \cdots, x_m) \mapsto (x_1,\cdots,x_{n-1},\sum_{k=n}^{m-i}{x_k}^2-\sum_{k=m-i+1}^{m}{x_k}^2)$$
 for an integer $0 \leq i(p) \leq \frac{m-n+1}{2}$.
\end{Def}

\begin{Prop}
\label{prop:1}
 For any singular point $p$ of the fold map in Definition \ref{def:1}, the $i(p)$ is unique {\rm (}$i(p)$ is called the {\rm index} of $p${\rm )}, the set of all singular points of a fixed index of the map is a closed smooth submanifold of dimension $n-1$ of the source manifold and the restriction map to the singular set is a smooth immersion of codimension $1$.
\end{Prop}
{\it Stable} maps are fundamental and important in various fields of the singularity theory. If we perturb a smooth map slightly or more precisely, we deform the map slightly in the space of smooth maps with the {\it Whitney $C^{\infty}$ topology} and the resulting map does not change from the original one modulo {\it $C^{\infty}$ equivalent} relations, which we explain later, then the map is said to be {\it stable}.
 A fold map is stable if and only if the restriction map has only normal crossings as self-crossings.
For fundamental stuffs and properties on fold maps and stable maps, see \cite{golubitskyguillemin} for example. For differential topological viewpoints of fold maps, see \cite{saeki} for example.
\subsubsection{Reeb spaces}
In studying fold maps, stable maps etc., {\it Reeb spaces} (\cite{reeb}) are fundamental and important. The {\it Reeb} space of a map $c:X \rightarrow Y$ is the space of all connected components of inverse images of $c$ and denoted by $W_c$. We denote the quotient map onto $W_c$ by $q_c:X \rightarrow W_c$. We can define the natural map uniquely $\bar{c}$ satisfying the relation $c=\bar{c} \circ q_c$.

 For example, the Reeb space is a graph if $c$ is a Morse function etc., a polyhedron of dimension equal to the dimension of the target manifold if it is a fold map and in considerable cases such as for stable maps, this holds (\cite{shiota}).

\subsubsection{Several explicit fold maps}
Stable Morse functions exist densely on any closed manifold. A closed manifold of dimension larger than $1$ admits a (stable) fold map into the plane if and only if the Euler number is even.
Existence of fold maps into general Euclidean spaces have been studied since Eliashberg's general
 studies \cite{eliashberg} \cite{eliashberg2} etc.. 

A fold map is said to be {\it special generic} if any singular point $p$ is of index $0$. For example, Morse functions with just $2$ singular points, characterizing spheres topologically (except $4$-dimensional exotic spheres), and canonical projections of unit spheres are stable and special generic: the singular value sets of the latter maps are embedded spheres. 

As an advanced result of \cite{saeki2}, every homotopy sphere of dimension $m>1$ not being an exotic $4$-dimensional sphere admits a special generic stable map into the plane whose singular value set is an embedded circle. 
As a more advanced result of the paper, a homotopy sphere of dimension $m>3$ admitting a special generic map into ${\mathbb{R}}^n$ ($n=m-3,m-2,m-1$) is diffeomorphic to $S^m$. In \cite{wrazidlo}, Wrazidlo has found more advanced restrictions on the differentiable structures of $7$ or higher dimensional homotopy spheres admitting special generic maps whose singular value sets are embedded spheres.

\begin{Fact}[\cite{saeki2}] 
\label{fact:1}
A manifold of dimension $m>1$ admits a special generic map into ${\mathbb{R}}^n$ satisfying the relation $m>n \geq 1$ if and only if $M$ is obtained by gluing the following two compact manifolds on the boundaries by a bundle isomorphism between
 the bundles whose fibers are $S^{m-n}$ defined on the boundaries in canonical ways. 
\begin{enumerate}
\item The total space of a linear $D^{m-n+1}$-bundle over the boundary of a compact $n$-dimensional manifold $P$ immersed into ${\mathbb{R}}^n$. 
\item The total space of a bundle over $P$ whose fiber is $S^{m-n}$.
\end{enumerate}
Thus, the Reeb space of a special generic map into an Euclidean space is a smooth compact manifold $P$, immersed into the target Euclidean space. 

\end{Fact}
%
%
%

Special generic maps are easy to handle and interesting objects from the viewpoint of the global singularity theory and differential topology.
However, the class of manifolds admitting special generic maps seem to be not so wide. In fact, several theorems on classifications of manifolds admitting special generic maps into fixed dimensional Euclidean spaces show this.
For example, manifolds whose dimensions are larger than $2$ represented as connected sums of total spaces of bundles whose fibers are almost-spheres over the circle are characterized as manifolds admitting special generic maps into the plane (\cite{saeki2}, \cite{saeki3} etc.). As a more general explicit case, closed and connected manifolds whose fundamental groups are free and whose dimensions are $m>3$ admitting special generic maps into ${\mathbb{R}}^3$ are characterized as manifolds represented as connected sums of total spaces of bundles whose fibers are standard spheres over the circle or $S^2$ under an appropriate assumption on $m$: $m=4,5,6$ for example (\cite{saeki3}, \cite{saekisakuma}, \cite{saekisakuma2} etc.). 
 
It is fundamental and important to construct various explicit fold maps other than these maps and obtain and investigate manifolds admitting such explicit maps. 
Thus fold maps satisfying weaker conditions and classes of fold maps giving manifolds of wider classes the source manifolds are important. As a simplest class, we introduce {\it round} fold maps introduced by the author based on the stream in \cite{kitazawa}, \cite{kitazawa2} and \cite{kitazawa3}.
 
Before the introduction, we define an equivalence relation on the family of smooth maps. Two smooth maps $c_1:X_1 \rightarrow Y_1$ and $c_2:X_2 \rightarrow Y_2$ are said to be {\it $C^{\infty}$ equivalent} if
a pair $({\phi}_X,{\phi}_Y)$ of diffeomorphisms satisfying the relation ${\phi}_Y \circ c_1=c_2 \circ {\phi}_X$ exists. We also say that $c_1$ is {\it $C^{\infty}$ equivalent to} $c_2$

\begin{Def}[\cite{kitazawa2}, \cite{kitazawa3}, \cite{kitazawa4} etc.]
\label{def:2}
$f:M \rightarrow {\mathbb{R}}^n$ is said to be a {\it round} fold map if either of the following hold.
\begin{enumerate}
\item $n=1$ holds and then $f$ is $C^{\infty}$ equivalent to
 a stable Morse function $f_0:M_0 \rightarrow {\mathbb{R}}$ on a closed manifold $M_0$ such that the following three hold.
\begin{enumerate}
\item $0$ is a regular value of $f_0$.  
\item Two Morse functions defined as ${f_0} {\mid}_{{f_0}^{-1}(-\infty,0]}$ and ${f_0} {\mid}_{{f_0}^{-1}[0,+\infty)}$ are $C^{\infty}$ equivalent.
\item $f_0(S(f_0))$ is the set of all integers whose absolute values are positive and not larger than a positive integer.  
\end{enumerate}
\item $n \geq 2$ holds and $f$ is $C^{\infty}$ equivalent to
 a fold map $f_0:M_0 \rightarrow {\mathbb{R}}^n$ on a closed manifold $M_0$ such that the following three hold.

\begin{enumerate}
\item The singular set $S(f_0)$ is a disjoint union of standard spheres whose dimensions are $n-1$ and consists of $l >0$ connected components.
\item The restriction map $f_0 {\mid}_{S(f_0)}$ is an embedding.
\item Let ${D^n}_r:=\{(x_1,\cdots,x_n) \in {\mathbb{R}}^n \mid {\sum}_{k=1}^{n}{x_k}^2 \leq r \}$. Then, the relation $f_0(S(f_0))={\sqcup}_{k=1}^{l} \partial {D^n}_k$ holds.  
\end{enumerate}
\end{enumerate}
\end{Def}

Stable Morse functions with just $2$ singular points and stable special generic maps on homotopy spheres before are simplest examples of round fold maps.

\begin{Ex}
\label{ex:1}
 (\cite{kitazawa},\cite{kitazawa2},\cite{kitazawa4} etc.)
Let $m>n$ be positive integers. Let $M$ be an $m$-dimensional closed and connected manifold and $\Sigma$ be an ($m-n$)-dimensional almost-sphere. Then the following are equivalent.
\begin{enumerate}
\item $M$ is the total space of a bundle over $S^n$ whose fiber is $\Sigma$.
\item 
Either of the following holds.
\begin{enumerate}
\item $n=1$ and $M$ admits a round fold map with just four singular points. The inverse images of regular values in the five connected components of the regular value set are $\emptyset$, $S^{m-n}$, $\Sigma \sqcup \Sigma$, $S^{m-n}$ and $\emptyset$, respectively.
\item $n \geq 2$ and $M$ admits a round fold map such that the singular set is the disjoint union of two copies of $S^{n-1}$. The inverse images of regular values in the three connected components of the regular value set are $\emptyset$, $S^{m-n}$ and $\Sigma \sqcup \Sigma$, respectively. Furthermore, on the complement of
the interior of an $n$-dimensional standard closed disc smoothly embedded in the connected component of the center of the regular value set and its inverse image, the map is $C^{\infty}$ equivalent to a product of a Morse function with just two singular points on the cylinder $\Sigma \times [-1,1]$ such that the boundary coincides with the inverse image of the minimum and the identity map ${\rm id}_{S^{n-1}}$.
\end{enumerate}
\end{enumerate}

Moreover, the Reeb space is obtained by attaching an $n$-dimensional closed disc to an ($n-1$)-dimensional standard sphere smoothly embedded in the interior of another $n$-dimensional closed disc on the boundary. Thus this polyhedron is simple homotopy equivalent to a bouquet of two copies of $S^n$.
\end{Ex}

For example, we can obtain several fold maps on $7$-dimensional exotic homotopy spheres constructed by Milnor in \cite{milnor} into ${\mathbb{R}}^4$. We cannot obtain a special generic map into ${\mathbb{R}}^n$ ($n=1,2,3,4,5,6$) on any $7$-dimensional exotic homotopy sphere. 

\subsection{Main stuffs and the content of this paper}
\subsubsection{Surgery operation to construct explicit fold maps more}

To obtain maps other than the presented fundamental maps explicitly and systematically, the author has introduced normal bubbling operations as surgery operations to stable fold maps in \cite{kitazawa5}. We remove the interior of a small closed tubular neighborhood of a closed and connected and orientable submanifold in the regular value of an original fold map and a connected component of its inverse image and after that, we attach a new map such that the singular value set is regarded as the boundary of the closed tubular neighborhood and in the interior of the target space and that the map obtained by the restriction to the singular set is an embedding. In this way we obtain a new stable fold map.
These operations are generalizations of {\it bubbling surgeries} by Kobayashi \cite{kobayashi3}: Kobayashi`s surgery is the case where the manifold is a point. 

\begin{Ex}
\label{ex:2}
The map of Example \ref{ex:1} is obtained by a bubbling surgery to a stable special generic map whose singular value set is a standard sphere.
\end{Ex}

\subsubsection{Homological properties of maps obtained by surgery operations and the content of this paper}
The author has constructed maps by finite iterations of such surgery operations through the following problem based on fundamental observations.
   
\begin{Prob*}
Let $R$ be a principal ideal domain and for a fold map $f:M \rightarrow N$ from a closed and connected manifold of dimension $m$ into an manifold without boundary of dimension $n$ satisfying the relation $m>n \geq 1$, let $f^{\prime}:{M}^{\prime} \rightarrow N$ be a fold map obtained by a finite iteration of normal bubbling operations to $f$. Then for any integer $0 \leq j \leq n$, we can set a finitely generated module $G_j$ over $R$ so
  that $G_0$ is a trivial $R$-module and that $G_n$ is not
 a trivial $R$-module and the module $H_j(W_{{f}^{\prime}};R)$ is isomorphic to $H_j(W_f;R) \oplus G_j$. 

 Conversely, for a suitable family $\{G_j\}_{j=0}^{n}$ of modules, can we construct
 a suitable map $f^{\prime}$ satisfying the condition by a finite iteration of normal bubbling operations starting from $f$?
\end{Prob*}
Example \ref{ex:1} or \ref{ex:2} accounts for the case where $G_j$ is zero for $1 \leq j \leq n-1$ and isomorphic to $\mathbb{Z}$ for $j=n$. In \cite{kitazawa5}, we have shown that we can realize the modules flexibly by explicit construction. Explicit related results will be presented in Propositions \ref{prop:2} and \ref{prop:3}. In \cite{kitazawa6}, we show more explicit restrictions: more precisely, we have investigated cases where the numbers of non-trivial $G_j$ are small. In addition, in the process, we have found new sufficient conditions to
 obtain $f^{\prime}$ from the original map $f$ we could not find in \cite{kitazawa5}.

In this paper, as a related study and an advanced version of the presented study of \cite{kitazawa6}, we present
 new general sufficient conditions on the groups $G_j$ via new methods: we apply elementary discussions on sequences of numbers and fundamental continuous or differentiable functions.

\subsubsection{The content of the paper}
The content of the paper is as the following. 
In the next section, we introduce and review {\it normal bubbling operations} based on \cite{kitazawa5}. The following
 will be also introduced as a key ingredient to know precise information of the source manifolds from the Reeb spaces.
For a stable fold map such that inverse images of regular values are disjoint unions of almost-spheres satisfying appropriate differential topological conditions, homology groups and homotopy groups of the source manifold and the Reeb space whose degrees are smaller than the difference of the dimensions of the source and the target manifolds are isomorphic. Special generic maps, maps presented in Example \ref{ex:1}, Fact \ref{fact:2} in the next section etc. satisfy the assumption of this fact. We can know homology and homotopy groups of the source manifolds from the Reeb spaces in these simple cases.

The last section is for main results. We present the new sufficient conditions for the groups $G_j$. Most of them will be obtained via elementary discussions on sequences of numbers and fundamental continuous or differentiable functions, which are new methods.

\subsection{A short remark on the content and acknowledgement}
\thanks{Closely related to the present paper, the author has presented preprints \cite{kitazawa5}, \cite{kitazawa6} etc.. However, we can read this paper 
even if we do not know the contents of them so much. Related stuffs we need will be presented in this paper.
 
The author is a member of and supported by the project Grant-in-Aid for Scientific Research (S) (17H06128 Principal Investigator: Osamu Saeki)
"Innovative research of geometric topology and singularities of differentiable mappings"

( 
https://kaken.nii.ac.jp/en/grant/KAKENHI-PROJECT-17H06128/
).}

\section{Normal bubbling operations}
\label{sec:2}
\begin{Def}[\cite{kitazawa5}]
\label{def:1}
Let $f:M \rightarrow N$ be a fold map from a closed and connected manifold of dimension $m$ into an manifold without boundary of dimension $n$ satisfying the relation $m>n \geq 1$, let $O$ be a connected component of the regular
 value set $N-f(S(f))$. Let $S$ be a connected and orientable closed submanifold of
 $O$ and $N(S)$, ${N(S)}_i$ and ${N(S)}_o$ be small closed tubular neighborhoods
 of $S$ in $O$ such that the relations ${N(S)}_i \subset N(S) \subset {N(S)}_o$ and ${N(S)}_i \subset {\rm Int }N(S)$ and $N(S) \subset {\rm Int} {N(S)}_o$ hold and they have sections seen as normal bundles over $S$: for example they are in an open set realized as an open submanifold of ${\mathbb{R}}^n$ ($N={\mathbb{R}}^n$ case for example: note that this condition is not included in the preprints \cite{kitazawa5} or \cite{kitazawa6} of the author due to my carelessness). Let
 $f^{-1}({N(S)}_o)$ have a connected component $P$ such that $f {\mid}_{P}$
 makes $P$ a bundle over ${N(S)}_o$.

Furthermore we assume that there exist an $m$-dimensional closed manifold $M^{\prime}$ and
 a fold map $f^{\prime}:M^{\prime} \rightarrow {\mathbb{R}}^n$
 satisfying the following.
\begin{enumerate}
\item $M-{\rm Int} P$ is a compact submanifold (with non-empty boundary) of $M^{\prime}$ of dimension $m$.
\item $f {\mid}_{M-{\rm Int} P}={f}^{\prime} {\mid}_{M-{\rm Int} P}$ holds.
\item ${f}^{\prime}(S({f}^{\prime}))$ is the disjoint union of $f(S(f))$ and $\partial N(S)$.
\item $(M^{\prime}-(M-P)) \bigcap {{f}^{\prime}}^{-1}({N(S)}_i)$ is empty or ${{f}^{\prime}} {\mid}_{(M^{\prime}-(M-P)) \bigcap {{f}^{\prime}}^{-1}({N(S)}_i)}$ makes $(M^{\prime}-(M-P)) \bigcap {f^{\prime}}^{-1}({N(S)}_i)$ a bundle over ${N(S)}_i$.
\end{enumerate}
 These assumptions allow us to consider the procedure of constructing $f^{\prime}$ from $f$. We
 call it a {\it normal bubbling operation} to $f$ and ${\bar{f}}^{-1}(S) \bigcap q_f(P)$, which is homeomorphic
 to $S$, the {\it generating manifold} of the normal bubbling operation. 
 
\begin{enumerate}
\item Let us suppose the following additional conditions.
\begin{enumerate}
\item
 ${{f}^{\prime}} {\mid}_{(M^{\prime}-(M-P)) \bigcap {f^{\prime}}^{-1}({N(S)}_i)}$ makes $(M^{\prime}-(M-P)) \bigcap {f^{\prime}}^{-1}({N(S)}_i)$ the disjoint union of two bundles over ${N(S)}_i$, then the procedure is called a {\it normal M-bubbling operation} to $f$. Note that the original inverse image having no singular points is represented as a connected sum of new two manifolds appearing as fibers of the two bundles.
\item ${{f}^{\prime}} {\mid}_{(M^{\prime}-(M-P)) \bigcap {f^{\prime}}^{-1}({N(S)}_i)}$ makes $(M^{\prime}-(M-P)) \bigcap {f^{\prime}}^{-1}({N(S)}_i)$ the disjoint union of two bundles over ${N(S)}_i$ and the fiber of one of the bundles is an almost-sphere, then the procedure is called a {\it normal S-bubbling operation} to $f$. Note that this operation is also a normal M-bubbling operation.
\end{enumerate}
\item As extra assumptions, if the following two hold, then the procedure is called a {\it trivial normal bubbling operation}. 
\begin{enumerate}
\item The map ${f^{\prime}} {\mid}_{{f^{\prime}}^{-1}({N(S)}_o-{\rm Int} {N(S)}_i)}$ is $C^{\infty}$ equivalent to the product of a Morse function with just one singular point and the identity
 map ${\rm id} {\mid}_{\partial N(S)}$.
\item The map ${f^{\prime}} {\mid}_{{{f}^{\prime}}^{-1}({N(S)}_i)}$ makes $(M^{\prime}-(M-P)) \bigcap {f^{\prime}}^{-1}({N(S)}_i)$ a trivial bundle over ${N(S)}_i$.
\end{enumerate}
\end{enumerate}
\end{Def}
\begin{Fact}[\cite{kitazawa2}, \cite{kitazawa4} etc.]
\label{fact:2}
Let $m$ and $n$ be positive integers satisfying the relation $m \geq 2n$ and let $M$ be an $m$-dimensional closed and connected manifold. Then the following are equivalent.
\begin{enumerate}
\item $M$ is represented as a connected sum of $l>0$ manifolds regarded as the total spaces bundles whose fibers are $S^{m-n}$ over $S^n$.
\item $M$ admits a round fold map obtained by $l$-times trivial normal S-bubbling operations starting from a canonical projection of a unit sphere, more generally, a special generic map on a homotopy sphere into the plane whose singular value set is an embedded circle or a higher dimensional version, such that inverse images of regular values are disjoint unions of standard spheres.
\end{enumerate}


\end{Fact}


Note that this can be regarded as an extension of Example \ref{ex:1} where the relation $m \geq 2n$ holds with $\Sigma=S^{m-n}$.The Reeb space is simple homotopy equivalent to a bouquet of $l$ copies of $S^n$.


In this paper, we only need the following trivial normal M-bubbling operations essentially as in \cite{kitazawa5}, \cite{kitazawa6} etc..  

\begin{Ex}
\label{ex:3}
Let $f:M \rightarrow N$ be a fold map from a closed and connected manifold of dimension $m$ into an manifold without boundary of dimension $n$ satisfying the relation $m>n \geq 1$, let $O$ be a connected component of the regular
 value set $N-f(S(f))$. Let $S$ be a connected and orientable closed submanifold of an open disc in $O$ such that there exists a connected component $P$ of $f^{-1}(S)$ and $f {\mid}_{P}$ gives a trivial bundle. Then we can consider a trivial normal M-bubbling operation whose generating manifold is $q_f(P)$ so that the pair of the resulting new two connected components of an inverse image having no singular points or the pair of the fibers of the resulting two bundles explained in Definition \ref{def:1} can be any pair of manifolds by a connected sum of which we obtain the original connected component of an inverse image having no singular points. Note that a trivial S-bubbling operation is a specific case.   
\end{Ex}

We present several results of \cite{kitazawa5} with sketches of proofs. 

\begin{Prop}[\cite{kitazawa5}]  
\label{prop:2}
Let $f:M \rightarrow N$ be a fold map from a closed and connected manifold of dimension $m$ into an manifold without boundary of dimension $n$ satisfying the relation $m>n \geq 1$, let $O$ be a connected component of the regular
 value set $N-f(S(f))$. Let $f^{\prime}:{M}^{\prime} \rightarrow N$ be a fold map obtained by a normal M-bubbling operation to $f$ and $S$ be the generating manifold of the normal M-bubbling
 operation satisfying $\bar{f}(S) \subset O$.  
Then, for any principal ideal domain $R$ and any integer $0\leq i<n$, we have

$$H_{i}(W_{{f}^{\prime}};R) \cong H_{i}(W_f;R) \oplus H_{i-(n-{dim} S)}(S;R)$$

and we also have $H_{n}(W_{{f}^{\prime}};R) \cong H_{n}(W_f;R) \oplus R$.
\end{Prop}

\begin{proof}[Sketch of the proof]
We can take a small closed tubular neighborhood, regarded as
 the total space of a linear $D^{n-\dim S}$-bundle over $S$. $W_{f^{\prime}}$ is regarded as a polyhedron obtained by attaching a manifold represented as the total space of a linear $S^{n-\dim S}$-bundle over $S$ by considering $D^{n-\dim S}$ in the beginning of
 this proof as a hemisphere of $S^{n-\dim S}$ and identifying the subspace obtained by restricting the
 space to fibers $D^{m-\dim S}$ with the original regular neighborhood. Note that the total space of the linear $S^{n-\dim S}$-bundle
 over $S$ is regard as a product bundle in knowing only the homology group of $W_{f^{\prime}}$. In fact, the bundle admits a section, corresponding
 to the submanifold $S$ and regarded as the image of the section obtained by taking the origin
 in each fiber $D^{n-\dim S_j} \subset S^{n-\dim S_j}$. 

By seeing the topologies of $W_f$ and $W_{f^{\prime}}$, we have the result.
\end{proof}

\begin{Cor}[\cite{kitazawa5}]
\label{cor:1}
In Problem in the introduction, The groups $G_{n-1}$ and $G_n$ must be free. 
\end{Cor}
\begin{proof}
$S$ is closed, connected and orientable and two groups $H_{\dim S-1}(S;\mathbb{Z})$ and $H_{\dim S}(S;\mathbb{Z})$ are free. This leads us to the desired results. 
\end{proof}

\begin{Prop}[\cite{kitazawa5}]  
\label{prop:3}
Let $R$ be a principal ideal domain. 
For any integer $0 \leq j \leq n$, we define $G_j$ as a free finitely generated $R$-module so
  that $G_0$ is a trivial $R$-module, that $G_n$ is not
 a trivial $R$-module and that the relation ${\Sigma}_{k=1}^{n-1}{{\rm rank}_R G_k} \leq {\rm rank}_R G_n$ holds. Then, by a finite iteration
 of normal M-bubbling {\rm (}S-bubbling{\rm )} operations starting from $f$, we obtain a fold map $f^{\prime}$ and $H_j(W_{{f}^{\prime}};R)$ is isomorphic to $H_j(W_f;R) \oplus G_j$. 
\end{Prop}
\begin{proof}
We can choose a family of standard spheres and points in $W_f-q_f(S(f))$ satisfying the following.
\begin{enumerate}
\item The family includes just ${\rm rank}_R G_j$ copies of $S^{n-j}$ for $1 \leq j \leq n-1$.
\item The family includes just ${\rm rank}_R G_n- {\Sigma}_{k=1}^{n-1}{{\rm rank}_R G_k}$.
\end{enumerate}
For the family of the spheres and the points, we can perform trivial normal S-bubbling operations whose generating manifolds are the chosen spheres or points one after another (we must take the points and spheres in open discs in $W_f-q_f(S(f))$ for example). Thus we have a desired map.
\end{proof}
Example \ref{ex:1} and Fact \ref{fact:2} account for the case $G_j=\{0\}$ ($0 \leq j \leq n-1$) of Proposition \ref{prop:3}.

The following is fundamental and important in knowing topological information of the source manifolds of explicit fold maps in the present paper. However, application of this is left to readers throughout this paper.

\label{prop:4}
\begin{Prop}[\cite{kitazawa}, \cite{kitazawa2}, \cite{kitazawa3}, \cite{saeki2}, \cite{saekisuzuoka} etc.]
\label{prop:4}
For a stable fold map on a closed and connected manifold of dimension $m$ into an $n$-dimensional manifold with no boundary obtained by a finite iteration of normal S-bubbling operations starting from a special generic map, let the relation $k=m-n>1$ hold. Then the quotient map onto the Reeb space induces isomorphisms of homology and homotopy groups of degree $l<k$.

As a specific case, if the map is special generic, then the quotient map onto the Reeb space induces isomorphisms of homology and homotopy groups of degree $l \leq k$ of the source manifold and those of the Reeb space.  
\end{Prop}
Maps in Example \ref{ex:1} or \ref{ex:2} and Fact \ref{fact:2} explain Proposition \ref{prop:4} well.
This holds for more general situations. See the cited articles,  \cite{kitazawa5} and \cite{kitazawa6}. 

\section{Main results and their proofs}
\label{sec:3}

A {\it homology group} of a topological space means a homology group of the space of a fixed degree.
\begin{Def}
In Problem in the introduction, we can naturally define the pair $(f,f^{\prime})$ of maps and the family of groups $\{G_j\}_{j=0}^{n}$. We call the pair $((f,f^{\prime}),\{G_j\})$ a {\it normal bubbling pair} and $\{G_j\}$ or a sequence $\{{G^{\prime}}_{j}\}_{j=0}^{n}$ of groups such that each pair $(G_j,{G^{\prime}}_j)$ of the corresponding groups are mutually isomorphic is said to be a sequence of groups {\it realized in a normal bubbling pair} or {\it R-NBP}. 
\end{Def}
Note that for the family of normal bubbling operations to get the map $f^{\prime}$ from $f$, it is sufficient to consider only normal M(S)-bubbling operations.  
\begin{Cor}
\label{cor:2}
Consider a normal bubbling pair $((f,f^{\prime}),\{G_j\}_{j=0}^{n})$. Let $S$ be a generating manifold such that the dimension is maximal among all the generating manifolds of the normal bubbling operations.
Thus, the minimal number $j=j_0$ such that $G_j$ is not trivial is $n-\dim S$, the group $G_{n-\dim S}$ is free and the relation ${\rm rank}_{\mathbb{Z}} G_{n-\dim S} \leq {\rm rank}_{\mathbb{Z}} G_{n}$ holds.
\end{Cor}
\begin{proof}
The time of normal M-bubbling operations whose generating manifolds are of dimension $\dim S$ is ${\rm rank}_{\mathbb{Z}} G_{n-\dim S}$ from Proposition \ref{prop:2} and $n-\dim S$ is the minimal number $j=j_0$ such that $G_j$ is not trivial. Furthermore, the group must be free. ${\rm rank}_{\mathbb{Z}} G_{n}$ is the time of all normal M-bubbling operations and from Corollary \ref{cor:1}, the group $G_n$ must be free. 
\end{proof}
In Corollary \ref{cor:2}, we call $j_0$ the {\it effective minimum} of a normal bubbling pair $((f,f^{\prime}),\{G_j\}_{j=0}^{n})$. We also abuse this terminology for a sequence of finitely commutative groups of a finite length.

As a natural question and for construction of maps by the surgery operations, we will attack the following fundamental problem.
\begin{Prob*}
Investigate sufficient conditions for $\{G_j\}$.
\end{Prob*} 
Several sufficient conditions have been found in \cite{kitazawa5} and \cite{kitazawa6} including Proposition \ref{prop:3}. 

We will find new conditions in the present paper in a new way based on fundamental theory of sequences of numbers and calculus. 

Let $k>0$ be an integer. A sequence $\{a_j\}_{j=1}^{k}$ of real numbers is said to be {\it strictly increasing} if for any pair $(j_1,j_2)$ of integers satisfying the inequality $0 \leq j_1<j_2 \leq n$, the inequality $a_{j_1} < a_{j_2}$ holds.

\begin{Thm}
\label{thm:1}
For any sequence $\{G_j\}_{j=0}^{n}$ of free finitely generated commutative groups such that for the effective minimum $j_0$ the subsequence $\{{\rm rank}_{\mathbb{Z}} G_j\}_{j=j_0}^{n}$ is strictly increasing. Then the given sequence is R-NBP.
\end{Thm}
\begin{proof}
Let $f:M \rightarrow N$ be a fold map from a closed and connected manifold of dimension $m$ into an manifold without boundary of dimension $n$ satisfying the relation $m>n \geq 1$. We choose a family of generating manifolds consisting of the following manifolds in an open ball in a connected component of $W_f-q_f(S(f))$.
\begin{enumerate}
\item A manifold represented as a connected sum of ${\rm rank}_{\mathbb{Z}} G_j$ copies of $S^{j-j_0} \times S^{n-j}$ for $j_0<j<j_0+\frac{n-j_0}{2}$ if $n-j_0$ is odd and a manifold represented as a connected sum of ${\rm rank}_{\mathbb{Z}} G_j$ copies of $S^{j-j_0} \times S^{n-j}$ for $j_0<j<j_0+\frac{n-j_0}{2}$ and $l$ copies of $S^{j-j_0} \times S^{n-j}$ for $j=j_0+\frac{n-j_0}{2}$ where $l$ is the largest integer satisfying the inequality $2l \leq {\rm rank}_{\mathbb{Z}} G_{j_0++\frac{n-j_0}{2}}$ if $n-j_0$ is even and positive.
\item ${\rm rank}_{\mathbb{Z}} G_n - 1$.
\end{enumerate}
By this, we can reduce the situation to a simpler similar situation. Proposition \ref{prop:2} implies that it is equivalent to consider the sequence $\{{{G}^{\prime}}_j\}_{j=0}^{n}$ of finitely generated commutative groups satisfying the following: we add several explanations to warrant this argument.
\begin{enumerate}
\item The effective minimum ${j_0}^{\prime}$ is $j_0+\frac{n-j_0+1}{2}$ if $n-j_0$ is odd and ${G^{\prime}}_{{j_0}^{\prime}}$ is isomorphic to ${\mathbb{Z}}^{{\rm rank}_{\mathbb{Z}} G_{j_0+\frac{n-j_0+1}{2}}-{\rm rank}_{\mathbb{Z}} G_{j_0+\frac{n-j_0-1}{2}}}$ and not trivial by the assumption that the subsequence $\{{\rm rank}_{\mathbb{Z}} G_j\}_{j=j_0}^{n}$ is strictly increasing.
\item The effective minimum ${j_0}^{\prime}$ is $j_0+\frac{n-j_0}{2}$ or $j_0+\frac{n-j_0}{2}+1$ if $n-j_0$ is even and positive. Moreover, in this case, the value is as the former and ${G^{\prime}}_{{j_0}^{\prime}}$ is isomorphic to $\mathbb{Z}$ if ${\rm rank}_{\mathbb{Z}} G_{j_0++\frac{n-j_0}{2}}$ is odd and the value is as the latter if the number is even. Moreover, in the latter case, ${G^{\prime}}_{{j_0}^{\prime}}$ is isomorphic to ${\mathbb{Z}}^{{\rm rank}_{\mathbb{Z}} G_{j_0+\frac{n-j_0}{2}+1}-{\rm rank}_{\mathbb{Z}} G_{j_0+\frac{n-j_0}{2}-1}}$ and of rank  $>1$ by the assumption that the subsequence $\{{\rm rank}_{\mathbb{Z}} G_j\}_{j=j_0}^{n}$ is strictly increasing.
\item $\{{\rm rank}_{\mathbb{Z}} {G^{\prime}}_{j}\}_{j={j_0}^{\prime}}^{n}$ is strictly increasing by the assumption on $\{{\rm rank}_{\mathbb{Z}}  G_j\}_{j=j_0}^{n}$.
\end{enumerate}
By an induction, we reduce the case to the case where the effective minimum is $n$. In this case, it is sufficient to take points as the generating manifolds. This completes the proof.
\end{proof}
\begin{Rem}
\label{rem:1}
We can weaken the assumption of Theorem \ref{thm:1}. We show examples of such cases. Readers can check that Theorem \ref{thm:1} is true for these cases.   
\begin{enumerate}
\item Let $n-j_0$ be odd. Let the subsequence $\{{\rm rank}_{\mathbb{Z}} G_j\}_{j=j_0+\frac{n-j_0-1}{2}}^{n}$ be strictly increasing. As a weaker assumption, we assume that the inequality ${\rm rank}_{\mathbb{Z}} G_{j+1} -{\rm rank}_{\mathbb{Z}} G_j \geq 0$ or the two inequalities ${\rm rank}_{\mathbb{Z}} G_{j+1} -{\rm rank}_{\mathbb{Z}} G_j <0$ and ${\rm rank}_{\mathbb{Z}} G_{j} -{\rm rank}_{\mathbb{Z}} G_{j+1} <{\rm rank}_{\mathbb{Z}} G_{n-(j-j_0)} -{\rm rank}_{\mathbb{Z}} G_{n-(j-j_0+1)}$ hold for $j_0 \leq j <j_0+\frac{n-j_0-1}{2}$.
\item Let $n-j_0$ be even. Let the subsequence $\{{\rm rank}_{\mathbb{Z}} G_j\}_{j=j_0+\frac{n-j_0}{2}-1}^{n}$ be strictly increasing. As a weaker assumption, we assume that the inequality ${\rm rank}_{\mathbb{Z}} G_{j+1} -{\rm rank}_{\mathbb{Z}} G_j \geq 0$ or the two inequalities ${\rm rank}_{\mathbb{Z}} G_{j+1} -{\rm rank}_{\mathbb{Z}} G_j <0$ and ${\rm rank}_{\mathbb{Z}} G_{j} -{\rm rank}_{\mathbb{Z}} G_{j+1} <{\rm rank}_{\mathbb{Z}} G_{n-(j-j_0)} -{\rm rank}_{\mathbb{Z}} G_{n-(j-j_0+1)}$ hold for $j_0 \leq j <j_0+\frac{n-j_0}{2}-1$.
\end{enumerate} 
\end{Rem}
\begin{Thm}
\label{thm:2}
Let $c$ be a differentiable function on $(0,+\infty)$ such
 that ${\lim}_{x \to +\infty} c^{\prime}(x)=+ \infty$ holds where $c^{\prime}$ is the derivative of $c$.
We define an sequence $\{G_j\}_{j=0}^{n}$ of free finitely generated commutative groups satisfying the following.
\begin{enumerate}
\item ${\rm rank}_{\mathbb{Z}} G_0=0$
\item ${\rm rank}_{\mathbb{Z}} G_j=0$ if $c(j)<0$.
\item ${\rm rank}_{\mathbb{Z}} G_j$ is the largest integer not larger than $c(j)$ if $c(j) \geq 0$.
\end{enumerate}
If $n$ is sufficiently large, then the given sequence is R-NBP.
\end{Thm}
\begin{proof}
From the assumption ${\lim}_{x \to +\infty} c^{\prime}(x)=+ \infty$, $c(x)>0$ holds for any sufficiently large $x>0$. In addition, from this assumption and the definition of the sequence of groups, the assumption in Remark \ref{rem:1} holds for any sufficiently large $n$. From Theorem \ref{thm:1} with Remark \ref{rem:1}, we immediately have the result.
\end{proof}

\begin{Ex}
For $c$ in Theorem \ref{thm:2}, we can take a polynomial function of degree $>1$ such that the coefficient of the top degree is positive, an exponential function $c(x)=a^x$ for $a>1$, $c(x):={\log}_{a}(x)$ for $a>1$ etc..
\end{Ex}

Last, we remark on polynomial functions of degree $1$ related
 to Proposition \ref{prop:3} and the new results.
\begin{Rem}
For a polynomial function of degree $1$ such that the coefficient of the top degree is positive, we may not apply Theorem \ref{thm:2} but we may apply Theorem \ref{thm:1} with Remark \ref{rem:1} to obtain a similar result. 

Moreover, we can easily check that this produces cases where the assumption of \ref{prop:3} does not hold for example.
\end{Rem}

\end{document}